\documentclass[11pt]{amsart}
\usepackage{ulem}  
\usepackage{color} 

\theoremstyle{plain}
\newtheorem{thm}{Theorem}[section]
\newtheorem{theorem}[thm]{Theorem}

\newtheorem{lemma}[thm]{Lemma}
\newtheorem{corollary}[thm]{Corollary}
\newtheorem{proposition}[thm]{Proposition}
\theoremstyle{definition}
\newtheorem{remark}[thm]{Remark}
\newtheorem{remarks}[thm]{Remarks}

\newtheorem{definition}[thm]{Definition}

\numberwithin{equation}{section}

\newcommand{\ga}[2]{\begin{gather}\label{#1}#2 \end{gather}}

\newcommand{\Pic}{{\rm Pic}}

\newcommand{\Spec}{{\rm Spec \,}}

\newcommand{\sB}{{\mathcal B}}

\newcommand{\sL}{{\mathcal L}}
\newcommand{\sM}{{\mathcal M}}

\newcommand{\sO}{{\mathcal O}}

\newcommand{\C}{{\mathbb C}}

\newcommand{\F}{{\mathbb F}}

\newcommand{\N}{{\mathbb N}}
\renewcommand{\P}{{\mathbb P}}
\newcommand{\Q}{{\mathbb Q}}
\newcommand{\R}{{\mathbb R}}

\newcommand{\Z}{{\mathbb Z}}

\title [non-liftable automorphisms]{Non-liftability of  automorphism groups  of a K3 surface in positive characteristic}
\author{H\'el\`ene Esnault and Keiji Oguiso} 
\address{Freie Universit\"at Berlin, Arnimallee 3, 14195, Berlin,  Germany}
\email{esnault@math.fu-berlin.de}
\address{Department of Mathematics, Osaka University, Toyonaka 560-0043, Osaka, Japan and Korea Institute for Advanced Study, Hoegiro 87, Seoul, 
133-722, Korea}
\email{oguiso@math.sci.osaka-u.ac.jp}
\thanks{The first  author is supported by  the Einstein program and the ERC
Advanced
Grant 226257. The second author is supported by JSPS Grant-in-Aid (S) No 25220701, JSPS Grant-in-Aid (S) No 22224001, JSPS Grant-in-Aid (B) No 22340009, and by KIAS Scholar Program.}

\date{ February 5th, 2015}
\begin{document}
\begin{abstract}
We show that  a characteristic $0$
 model $X_R\to \Spec R$, with Picard number $1$ over a geometric generic point,   of a  K3 surface in characteristic $p\ge 3$, essentially kills  all automorphisms (Theorem~\ref{cor:entropy2}). We show that there is an  explicitely constructed automorphism on a supersingular  K3 
surface in characteristic $3$, which has positive entropy, the logarithm of a Salem number of degree $22$ 
 (Theorem~\ref{NoLiftPosEntropy}). In particular it does not lift to characteristic $0$. In addition, we show that in any large characteristic,  there is an automorphism of a supersingular K3  which has positive entropy and does not lift to characteristic $0$ (Theorem~\ref{thm:beta}).
\end{abstract}
\maketitle

\section{Introduction}\label{intro}

Let $X$ be a K3 surface over an algebraically closed field $k$ of characteristic $p>0$. A classical theorem \cite{Del81} asserts that the formal universal deformation space $\hat S$  of $X$  is unobstructed, and is formally smooth of dimension $20$ over $W(k)$, the ring of Witt vectors of $k$.  Moreover, the closed formal  subscheme of $\hat S$ parametrizing the locus $\hat \Sigma(X, L)$  over which a given line bundle $L$ on $X$ lifts, is  a  hypersurface, flat over $W(k)$. The aim of our article is to understand  conditions for automorphisms of $X$  to be or not to be liftable to a proper  model $X_R\to \Spec R$ of $X$, where $R$ is a discrete valuation ring such that $\Spec R\to \hat S$ dominates $\Spec W(k)$.  Said in words, we study conditions on automorphisms of $X$   to lift to characteristic $0$, or not.  One motivation for this study is the observation that the crystalline classes of graphs of automorphisms on a positive characteristic K3 surface obey the Fontaine-Mazur  $p$-adic variational Hodge conjecture as expressed in 
\cite[Conj.~1.2]{BlEsKe14} (see Remark~\ref{rmk:hodge}).

\medskip

 Our main results are Theorem~\ref{cor:entropy2}, Theorem~\ref{NoLiftPosEntropy} and Theorem~\ref{thm:beta}. Simplified versions are Theorems~\ref{IntroThm2},~\ref{IntroThm3},~\ref{IntroThm4} explained below. For the discussion in the introduction, we assume for simplicity that $p\ge 3$.

\medskip

Recall one has  a natural  injective specialization  homomomorphism 
$$ \iota: {\rm Aut}^e\, (X_{\bar K}/\bar K) \to {\rm Aut}\, (X/k)\,\, $$ 
where  $K$ is the field of fractions of $R$ and ${\rm Aut}^e(X_{\bar K}/\bar K) \subset {\rm Aut}(X_{\bar K}/\bar K)$ is the subgroup of automorphisms which lift to some proper model $X_R\to {\rm Spec} \ R$.
We say  $f \in {\rm Aut}\, (X/k)$ is not geometrically liftable to characteristic $0$  if it  is not in the image
of $\iota$
  (see Section~\ref{sec:notations} for details). 

\medskip

In the complex case, for any K3 surface $M$ and any given line bundles $L_i$ ($1 \le i \le d \le 19$) on $M$  for which  the  $ L_i,  \ i=1, \ldots, d$ are part of a $\Z$-basis of ${\rm Pic}\, (M)$,  there is  a smooth proper small deformation 
$$\psi : (\sM, \sL_1, \ldots , \sL_d) \to \Delta$$
of $(M, L_1, \ldots , L_d)$, where $\Delta$ is the analytic disc,  such that ${\rm Pic}\, (\sM_t) = \langle (\sL_1)_t, \ldots , (\sL_d)_t \rangle \simeq \Z^d$ for  a point $t \in \Delta,$ generic  in the  complex analytic sense.
The proof goes via a study of the period map. 
 In particular, if $d=1$ and $L_1$ is ample primitive, then $\psi : \sM \to \Delta$ is also projective and ${\rm Pic}\, (\sM_t) = \Z (\sL_1)_t$ for generic $t$. As a consequence,  the specialization homomorphism has very small image even if ${\rm Aut}\, (M)$ is very large. In other words, interesting automorphisms disappear on the generic fiber $\sM_t$  (see eg. \cite{Og03}).

\medskip 

Our first aim is to show the analogous results Theorems ~\ref{IntroThm1}, ~\ref{IntroThm2}  on  liftings  from characteristic $p$ to characteristic $0$  (see  Theorems~\ref{thm:one},  ~\ref{cor}, 
~\ref{cor:entropy2} for more precise statements):

\begin{theorem} \label{IntroThm1} If $p\ge 3$, 
there is a discrete valuation ring $R$, finite over the ring of Witt vectors $W(k)$, together with a 
projective model $X_R\to \Spec R$, such that the Picard rank of $X_{ \bar K}$ is $1$. 
\end{theorem}
The proof is given in \cite[App.~A]{LieOls11}. The proof in {\it loc.cit.} relies on  \cite{Ogu79}, \cite{Ogu83} and on the properties of the stack parametrizing deformations of a K3 surface together with line bundles.  In Section~\ref{sec:lifting}, we sketch another proof, relying on 
 \cite{Del81}. Theorem~\ref{cor}  has its own interest. 

As an immediate but remarkable consequence of Theorem~\ref{IntroThm1}, we obtain:

\begin{theorem} \label{IntroThm2} If $p\ge 3$, there is a projective model $X_R\to \Spec R$ 
such that no subgroup $G \subset {\rm Aut}\, (X)$, except for 
$G = \{ {\rm id}_X \}$ is geometrically liftable to $X_R\to \Spec R$, unless 
${\rm Pic}\, (X) = \Z \cdot  H$ with  self-intersection number $(H^2) = 2$.  
\end{theorem}
See also Theorem ~\ref{cor:entropy2} 2) for the exceptional case. We also note that when $X$ is not supersingular, Theorem~\ref{IntroThm2} is in  sharp contrast to the model constructed by Lieblich and Maulik \cite{LieMau11}, to show the  Kamawata-Morrison Cone Conjecture for K3 surfaces in positive characteristic.  We prove Theorem~\ref{IntroThm2} in Section~\ref{sec:Nolifting}.

\medskip

The second aim of our article is to show the richness of automophisms of supersingular K3 surfaces of Artin invariant $1$, in  view of  the non-liftability problem. Supersingular K3 surfaces of Artin invariant $1$ are unique up to isomorphisms, for each field $k$ of characteristic $p >0$. They are the most special K3 surfaces (see Section ~\ref{sec:NonLiftable} for a brief review). We denote them by $X(p)$. Recently,  several interesting aspects of automorphisms of $X(p)$ for various $p$ were studied (\cite{DK09}, \cite{DK09-2}, \cite{Sh13}, \cite{KS12} and references therein). The notion of entropy is classical in the complex case, and is of topological nature (see e.g. \cite{Og14} and 
references therein); it has been  introduced   in  \cite{EsnSri13} in positive characteristic. The positivity of entropy is  a numerical  measure of complexity or richness of automorphisms in any characteristic (see Section \ref{ss:K3}). We note here that an automorphism of positive entropy is necessarily of infinite order, but it is a stronger constraint as  there are many automorphisms of infinite order with null-entropy. 

\medskip

Our next main results are Theorems ~\ref{IntroThm3} and  ~\ref{IntroThm4}, showing the richness of automorphisms of supersingular K3 surfaces of Artin invariant $1$:

\begin{theorem} \label{IntroThm3}
There is an $f \in {\rm Aut}\, (X(3))$ of positive entropy such that for all $n\in \Z\setminus \{0\}$,   $f^n$ 
is  not geometrically  liftable to characteristic zero. The entropy of $f$ is the logarithm of a Salem number of degree $22$.  \end{theorem}

This is the  first explicit example of an automorphism  of positive entropy which can  never be lifted to characteristic zero. In characteristic $2$, (non-explicit) examples have been constructed in \cite{BC13} (see \cite[Thm.~4.2]{EOY14} for a slight  clarification). 
See Theorem~\ref{NoLiftPosEntropy} for the precise statement and Section~\ref{sec:Salem} for the definition of Salem numbers. 

\medskip

Recall that there are  non-projective  complex K3 surfaces with an automorphism for which  the entropy  is the logarithm of a Salem number of degree $22$ (\cite{Mc02}). However,  the  entropy of an automorphism of a  projective K3 surface over a field of characteristic zero is either zero or the logarithm of Salem number of degree $\le 20$ (\cite{Mc02}, \cite{Mc13}). 
In particular, the Salem number we construct in Theorem~\ref{IntroThm3} can not be the Salem number associated to the entropy of a projective K3 surface in characteristic $0$.

Our construction is entirely based on a result of Kondo and Shimada \cite{KS12} and is mildly supported by a  {\tt Mathematica} computation. It would be nicer if one could find a more conceptual reason for the existence. 

\medskip

Finally we show:
\begin{theorem}\label{IntroThm4} 
For $p$ large, there is an automorphism  of $X(p)$, of positive entropy, which is not geometrically  liftable to characteristic $0$.
\end{theorem}

See also Theorem~\ref{thm:beta} for a more general statement. 
Our construction is based on Jang's result \cite[Thm.~3.3]{Jan14} together 
with 
 a  result on the Mordell-Weil groups of elliptic fibrations 
due to Shioda \cite{Sh90}. 
In the proof, one shows a way to construct  an automorphism of positive entropy out of  those of null-entropy. This might have  an  interest on its own.
\footnote{Since the construction presented in this article has been written, using  Shioda's work on the existence of different elliptic structures with rich N\'eron-Severi groups  on the surfaces $X(p)$ , it is shown in \cite[Thm.~1.1]{EOY14} that for almost all $p$, $X(p)$ carries an automorphism, the entropy of which is the logarithm of a Salem number of degree $22$.}
We prove Theorem~\ref{IntroThm3}, Theorem~\ref{IntroThm4} in Sections~\ref{sec:NonLiftable},
 ~\ref{sec:EllipticK3}.

\medskip

\noindent
{\bf Ackowledgement.} \\[.2cm]
We thank  Takao Iohara for kindly accepting to check again our {\tt Mathematica}  computation  using  {\tt Maxima}. We thank Gabriel F\"urstenheim for asking  us for concrete non-liftable automorphisms, which prompted the starting point of our work.  We thank Luc Illusie and Arthur Ogus for a friendly and  interesting discussion on the articles \cite{Del81} and \cite{Ogu79}.  We thank Matthias Sch\"utt for carefully reading   the first version,  
Junmeyong Jang for mentioning \cite{Jan14} to us, which enabled us to find Theorem ~\ref{IntroThm4}, and  Olivier Benoist for mentioning \cite{LieOls11} to us. Finally we thank the referee for a friendly and careful reading of our article which enabled us to improve it.

\section{Lifting automorphisms of K3 surfaces: notations and formulation of the  problems} \label{sec:notations}
\noindent
We introduce some notations and formulate the main questions addressed in our article. 
\subsection{Models and Lifts} \label{ss:models}
 Let $M$ be a  proper   variety defined over a perfect characteristic $p>0$ field $k$, and $R$ be a discrete valuation ring (in the sequel abbreviated as DVR)  with residue field $k$ and  field  of fractions  $K = {\rm Frac}\, (R)$. A {\it model of $M$ over $\Spec R$} is a proper flat morphism  of schemes $M_R\to \Spec R$ lifting $M \to \Spec k$. If $M$ is smooth, then  a lift  $M_R\to \Spec R$   is a model if and only if $M_R\to \Spec R$ is smooth. We call a model $M_R\to \Spec R$ a {\it lift to characteristic zero} if $K$ is of characteristic zero. If $M_R\to \Spec R$ is a model of $X$,  and $L\supset K$ is any field extension, we say that $M_L=M_R\otimes_R L$  is a {\it lift} of $M$ to $L$. 
 If $K$ has characteristic $0$, we say $M_L$ is a lift of $M$ {\it to characteristic} $0$. We call a lift $M_R\to \Spec R$ to characteristic zero a {\it projective model} of $M$ if $M_R\to \Spec R$ is projective.

\subsection{Automorphisms} \label{ss:auto}

Let $M$ and $S$ be schemes and $\varphi : M \to S$ be a morphism.
We denote by ${\rm Aut}\, (M/S)$ the {\it group of automorphisms of}  $M$ over $S$.  When $\varphi$ is flat projective, ${\rm Aut}\, (M/S)$  is the group of $S$-points of a group scheme representing the ${\rm Aut}(-/S)$-functor, but we will just use the abstract group ${\rm Aut}\, (M/S)$. If $S$ is the spectrum of a ring $R$, we also write ${\rm Aut}(M/R)$. Finally for $S=\Spec k$, we write ${\rm Aut}(M)$ instead of ${\rm Aut}(M/k)$ if there is no danger of confusion.

\subsection{K3 surfaces and automorphisms.} \label{ss:K3} 
A {\it K3 surface} $V$  {\it over a field} $F$ is a smooth projective geometrically irreducible $2$-dimensional variety defined over $F$ such that $H^1(V, {\mathcal O}_V) = 0$ and the dualizing sheaf $\omega_V$ is trivial, i.e., $\omega_V \simeq {\mathcal O}_V$. As  a smooth proper scheme  of dimension $2$ over a field is necessarily projective, we stick to the notion of 'projective' surface rather than 'algebraic' surface.

\medskip

 {\it Thoughout this article,   $X $ is a K3 surface over an algebraically closed field   $k$  of characteristic 
$p >0$}.  A model $X_R\to \Spec R$ of a K3 surface has the property that $X_K/K$ is a $K3$ surface.

\medskip

 The {\it  N\'eron-Severi group  ${\rm NS}\, (X)$} of $X$  is isomorphic to 
${\rm Pic}\, (X)$, and it is a free $\Z$-module 
of finite rank. The rank is called the {\it  Picard number} of $X$ and denoted by $\rho(X)$. We have $1 \le \rho(X) \le 22$.   The Hodge index theorem implies that the  N\'eron-Severi group ${\rm NS}\, (X)$ 
is an even  hyperbolic lattice with respect to the intersection from $(*, **)$,  i.e., $(*, **) \in  {\rm Sym}^2 ({\rm NS}(X)^\vee)$,  
 of signature $(1, \rho(X) -1)$ on ${\rm NS}\, (X) \otimes_{\Z} \R$. In  addition, 
$(x^2) := (x, x) \in 2 \Z$ for all $x \in {\rm NS}\, (X)$, as by the Riemann-Roch theorem $x^2=2( \chi(X, x) -\chi(X, \mathcal{O}_X))$.  We denote the group of isometries of $({\rm NS}\, (X), (*, **))$ by ${\rm O}\, ({\rm NS}\, (X))$.

\medskip

The action by pull-back of line bundles $L\mapsto f^*L$ defines a contravariant representation
$${\rm Aut}\, (X) \to {\rm O}\, ({\rm NS}\, (X)).$$ Let $f \in {\rm Aut}\, (X)$. The   {\it spectral radius} of $f^* \in {\rm O}\, ({\rm NS}\, (X))$, denoted by ${\rm sp}\, (f)$, is the maximum of the absolute values of eigenvalues of $f^* \otimes {\rm id}_{\C} \vert_{{\rm NS}\, (X) \otimes \C}$. Here and hereafter for a complex number $\alpha = a + b\sqrt{-1}$ ($a, b \in \R$) the absolute value $\vert \alpha \vert$ of $\alpha$ is the non-negative real number 
$$\vert \alpha \vert = \sqrt{a^2 + b^2}.\,\, $$ 
 One has  ${\det}\, f^* = \pm 1$. Thus, ${\rm sp}\, (f) \ge 1$. One defines $f$ to be of {\it positive entropy} (resp. of {\it null-entropy}) if 
${\rm sp}\, (f) > 1$ (resp. ${\rm sp}\, (f) =1$). The entropy of $f$ is defined by 
$$h(f) = \log\, {\rm sp}\, (f)\,\, .$$
This definition is in fact equivalent to the one obtained by first defining the entropy as the natural logarithm of the maximum of absolute values of the eigenvalues of $f^*$ acting of the $\ell$-adic cohomology ring $\oplus_i H^i(X, \mathbb{Q}_\ell)$,  with respect to all complex embeddings of the eigenvalues, as it is shown in \cite{ EsnSri13}  that regardless of the choice of the complex embedding, this maximum is taken on the N\'eron-Severi group.
This is also consistent with  the notion of entropy for complex projective K3 surfaces. Note that if $f$ is of positive entropy, then $f$ is of infinite order, while the converse is not true in general. Let $G \subset {\rm Aut}\, (X)$ be a subgroup. We call $G$ {\it of null-entropy} (resp. {\it of positive entropy}) if all the elements of $G$  are of null-entropy (resp. some element of $G$ is of positive entropy).

 \subsection{Specializations.}  \label{ss:sp}
 Let $X_R\to \Spec R$ be a smooth proper  morphism. 
 Recall  (\cite[X,~App.]{SGA6}) that  one has a {\it specialization homomorphism }
 $sp: \Pic(X_{\bar K})\to \Pic(X)$ {\it on the Picard group},  which is defined as follows:  any $\sL_{\bar K}\in \Pic(X_{\bar K})$ is defined over a finite extension $L\supset K$, $L\subset \bar K$, so $\sL_{\bar K}=\sL_L\otimes_L \bar K$. Let $R_L\subset L$ be the ring of integers. The restriction homomorphism $\Pic(X_{R_L})\to \Pic(X_L)$ is an isomorphism as $X_{R_L}$ is smooth. So $\sL_L=\sL_{R_L} \otimes L$. Then the specialization of $\sL_{\bar K}$ is $\sL_{R_L}\otimes k$.   
 
 The specialization factors through the N\'eron-Severi group $sp_{NS}: NS(X_{\bar K})\to NS(X_k)$ and through the N\'eron-severi group modulo torsion $sp_{NS/{\rm torsion}}:  NS(X_{\bar K})/{\rm torsion} \to NS(X_k)/{\rm torsion}$. Then $sp_{NS/{\rm torsion}}$ is injective, as $sp$ is compatible with the  injections  $NS(X_{\bar K})/{\rm torsion}\to H^2(X_{\bar K}, \Q_\ell(1)), \  NS(X)/{\rm torsion} \to H^2(X, \Q_\ell(1))$ defined by the Chern class  and the specialization
 $H^2(X_{\bar K}, \Q_\ell(1))\to H^2(X, \Q_\ell(1))$
  on $\ell$-adic cohomology, which is an
 isomorphism
  by the smooth proper base change theorem (\cite[V,~Thm.~3.1]{SGA4.5}).
 
 \medskip
 
 Let $X_R\to \Spec R$ be a  model of a K3 surface $X$.
 Then  one has a {\it restriction homomorphism  } ${\rm Aut}(X_R/R)\to {\rm Aut}(X)$. 
 Let us define the subset ${\rm Aut}^e (X_{\bar K}/{\bar K})\subset {\rm Aut}(X_{\bar K}/{\bar K})$ consisting of those automorphisms which lift to some model $X_R\to {\rm Spec} \ R$. (Here {\it e}  stands  for extendable).   It is clearly a subgroup, where the group law is define after  finite base change.  Then the restriction homomorphism yields a {\it specialization homomorphism} 
 \ga{}{\iota: {\rm Aut}^e(X_{\bar K}/\bar K)\to {\rm Aut}(X/k) . \notag} 
Moreover, $sp$ is equivariant under $\iota$.
   In addition,  as automorphisms are recognized on the associated formal scheme, and $H^0(X, T_{X/k})=0$, the specialization homomorphism 
$\iota$ is injective (see \cite[~Lem.~2.3]{LieMau11}) .
 
 We call  $f$  {\it geometrically liftable} if 
it is in the image of the specialization homomorphism $\iota$.
One  similarly defines  geometric liftability of a  subgroup $G\subset {\rm Aut}(X)$.

 \begin{remark}\label{biraut} It is natural to ask whether the subgroup ${\rm Aut}^e(X_{\bar K}/\bar K) \subset {\rm Aut}(X_{\bar K}/\bar K)$  is a strict subgroup.  We  give an explicit example for which  ${\rm Aut}^e(X_{\bar K}/\bar K) \not= {\rm Aut}(X_{\bar K}/\bar K)$ in 
Theorem~\ref{NoLiftTwoDegreeTwo}.
  Any $f \in {\rm Aut}\, (X_{\bar K}/\bar K)$ extends uniquely to $\tilde{f} \in {\rm Bir}\, (X_{R_L}/R_{L})$, the group of birational automorphisms, where $L$, with $K \subset L \subset \bar{K}$,  is a field of definition of $f$.   As $X$ is regular, if $\tilde{f}$ is not regular, then there is a $1$-dimensional subscheme $C\subset X$ such that $\tilde{f}$ is well defined as a morphism $\tilde{f}: X_{R_L}\setminus C\to X_{R_L}$, but the morphism does not necessarily extend to $X_{R_L}$. Finally this implies that for any field extension $L'\supset L$, $L'\subset \bar K$, the base changed morphism   $\tilde{f}\otimes R_{L'}:  X_{R_{L'}}\setminus C\to X_{R_{L'}}$ does not extend either.
 See \cite[Prop.~4.1]{LieMat14} for related phenomena. \end{remark}

\section{Automorphisms of even hyperbolic lattices} \label{sec:Salem}

We call a polynomial $P(x) \in \Z [x]$ a {\it Salem polynomial} if it is irreducible, monic,  of even degree $2d \ge 2$ and the complex zeroes of $P(x)$ are of the form ($1 \le i \le d-1$):
$$a > 1\,\, ,\,\, 0 < a^{-1} < 1\,\, ,\,\, \alpha_i, \overline{\alpha}_i \in S^1 := 
\{z \in \C\, \vert\, \vert z \vert = 1\} \setminus \{\pm 1\}\,\, .$$

\begin{proposition}\label{salem} Let $r$ be a positive integer and $L = (\Z^r, (*,**) \in {\rm Sym}^2 (\Z^r)^{\vee})$ be a hyperbolic lattice, i.e., the bilinear form $(*,**)$ is non-degenerate of signature $(1, r-1)$. Let $C := \{x \in L \otimes \R\, \vert\, (x^2) > 0\}$. Then $C$ has exactly two connected components, say $C^0$ and $-C^0$. Let $f \in {\rm Aut}\, L$ and assume that $f(C^0) \subset C^0$. Then, the characteristic polynomial of $f$ is the product of cyclotomic polynomials and at most one Salem polynomial. In particular,  when $f(C^0) \subset C^0$,
 the characteristic polynomial of $f$ is the product of cyclotomic polynomials if and only if $f$ is of null-entropy  if and ony if  $f$ is quasi-unipotent, 
  i.e., all the eigenvalues of $f^n$ are $1$ for some positive integer $n$.
\end{proposition}    
\begin{proof} This is well-known and essentially due to McMullen \cite{Mc02}. See also \cite{Og10}.
\end{proof} 
One way for an automorphism $f$  of an hyperbolic lattice to perserve $C^0$ is to fulfill $f(e)=e$ for a non-zero isotropic vector $e$.
\begin{remark}\label{entropy} 
 For $f$ as in Proposition~\ref{salem}, we define (by a slight abuse of notation)  $f$  to be  of positive entropy (resp. of null entropy) if ${\rm sp}\, (f) > 1$ (resp. ${\rm sp}\, (f) =1$). Thus $f$ is of positive entropy (resp. of null entropy) if and only if the characteristic polynomial of $f$ has a  Salem factor (then exactly one)  (resp. only cyclotomic factors). 
\end{remark}

\begin{proposition}\label{nullentropy} Let  $L$ be as in Proposition~\ref{salem} and $f$ be in ${\rm Aut}(L)$. 
  Assume that there is $e \in L \setminus \{ 0\}$ 
such that $f(e) = e$ with $(e^2) :=(e,e) = 0$. Then the characteristic polynomial of $f$ is the product of cyclotomic polynomials.  
\end{proposition}    

\begin{proof} 

We may assume without loss of generality that $e$ is primitive 
in the sense that $e$ is a part of $\Z$-basis of $L$. By the assumption, $f$ acts on the flag  $\Z e \subset (\Z e)^{\perp}$ and hence induces an automorphism $\bar f$ of   
$$\overline{L} := (\Z e)^{\perp}/\Z e\,\, .$$
The bilinear form of $L$ induces a bilinear form of 
$\overline{L}$ of signature $(0, r-2)$, i.e., $\overline{L}$ is  negative definite or $\{0 \}$. 
If  $ \overline{L} \neq 0$, 
 the eigenvalues of $\bar f$ on $\overline{L}$ are of absolute value $1$. Here we use the well-known fact that eigenvalues of a real orthogonal matrix are of absolute value $1$. Combining this with $f(e) = e$, we find that the eigenvalues of $f$ are of absolute value $1$ except perhaps one eigenvalue counted with multiplicities. Note that $\det\, f = \pm 1$, as 
an automorphism of free $\Z$-module $L$ of finite positive rank. Hence the last  eigenvalue is also of absolute value $1$. Since $f(C^0) \subset C^0$ as $f(e)= e$, this implies the result by Proposition~\ref{salem}.
\end{proof} 

\section{Lifting to characteristic $0$  K3 surfaces with Picard number one} \label{sec:lifting}

Let $M$ be a complex projective  K3 surface with a primitive ample line bundle $H$. Recall   $H$ is said to be {\it primitive}  if it is a part of $\Z$-basis of the finitely generated free $\Z$-module ${\rm Pic}\, (M) = {\rm NS}\, (M)$. 
 It is well-known that any generic  fiber ${\mathcal X}_t$ of the Kuranishi family $ \kappa: ({\mathcal X}, {\mathcal H}) \to {\mathcal K}$ of $(X, H)$ is of Picard number $1$. Here 'generic' is to be understood in the complex analytic sense.  In particular, restricting $\kappa$ to the germ of a smooth curve $\sB \subset \mathcal{K}$ containing a generic point, the kernel of the Gau{\ss}-Manin connection on $H^2$ on $\sB$ is spanned by the first de Rham Chern class of a generator of ${\rm NS}\,(\mathcal{X}_t)$.
  
  \medskip
  
In \cite[App.~A]{LieOls11},
   M. Lieblich and M. Olsson, 
   prove the analogous result on Picard  rank $1$ lifts to characteristic $0$ of characteristic $p\ge 3$ K3 surfaces.

\begin{theorem} \label{thm:one}
Let $X$ be a K3 surface defined over 
 an algebraically  closed field   $k$  of characteristic 
$p >0$, where $p>2$ if $X$ is  Artin-supersingular.
Then there is a discrete valuation ring $R$, finite over the ring of Witt vectors $W(k)$, together with a 
projective model $X_R\to \Spec R$, such that the Picard rank of $X_{ \bar K}$ is $1$, where $ K={\rm Frac}\,(R)$ and $\bar K\supset K$ is an algebraic closure. 
\end{theorem}
Their proof  (in $p \ge 3$) relies on \cite{Ogu79} and \cite{Ogu83} and properties of stacks for pairs of K3 surfaces together with line bundles. 

\medskip 
 
 For a line bundle $L$ on $X$, we denote by $c_1^{\rm Hodge}(L)\in H^1(X, \Omega^1_X)$ its Hodge Chern class.
\begin{proposition} \label{existence}
Let $X$ be a K3 surface  defined over
 an algebraically  closed field   $k$  of characteristic 
$p >0$,
with $p>2$  if     $X$ is  Artin-supersingular.  Then there is an ample primitive   line bundle $L$ such that $c_1^{\rm Hodge}(L)\neq 0$.  

\end{proposition}

\begin{proof}
Unfortunately, we can not prove it directly, this is the reason for the restriction on $p$.  If $X$ is not Artin-supersingular, then by \cite[Prop.~10.3]{GK00}, $$c_1^{{\rm Hodge}}: \Pic(X)/p\Pic(X)  \to H^1(X, \Omega^1_{X/k})$$ 
is injective.  Else, due to our assumption, it is Shioda-supersingular, 
thus by 
\cite[Prop.~11.9]{GK00}, $c_1^{{\rm Hodge}}$ is not identically zero.  Once one line bundle $M$ fulfills $c_1^{{\rm Hodge}}(M)\neq 0$, then given any ample line bundle $H$, $0\neq c_1^{{\rm Hodge}}(M)=c_1^{{\rm Hodge}}(M+mpH)$ for any integer $m$, and for $m$ large, $M+pmH$ is ample.  The $\Z$-module
$\Q\cdot ( M+pmH)\cap NS(X)$  of rank $1$ has a generator $L$ such $(M+mH)=aL$ where $a \in \N\setminus \{0\}$.
 Then $L$
is primitive, ample, and fulfills $c_1^{\rm Hodge}(L)\neq 0$.

\end{proof}

In fact, one can prove Theorem~\ref{thm:one} by using \cite{Del81} and Proposition~\ref{existence} only. Indeed, for $L$ as in Proposition~\ref{existence}, one shows (\cite[Prop.~2.2]{Ogu79} and \cite[Lem.~4.3]{LieMau11}) that the formal hypersurface $\Sigma(X,L)\subset \hat S={\rm Spf} \ W[[t_1,\ldots, t_{20}]]$, which is defined as the solution to the deformation functor of the pair $ (X,L)$, while $\hat S$ is defined as the solution of the deformation functor of $X$, is formally smooth.

Let $\hat X$ be the formal universal K3 surface over $\hat S$. Set $\hat Y=\hat X \times_{\hat S} \Sigma(X,L)$ and let $\sL$ be the formal universal line bundle on $\hat Y$ lifting $L$ on $X$. Then the de Rham class $c_1^{DR}(\sL)\in H^2_{DR}(\hat Y/\Sigma(X,L))$ is non-zero as its restriction in $H^1(X, \Omega^1_X)$ is non-zero.

  Then one shows, using  the injectivity of the crystalline Chern class map ${\rm Pic}(X)\to H^2_{DR}(\hat X_W/W)$ (\cite[2.10]{Del81})
\begin{theorem} \label{cor}
The kernel of the Gau{\ss}-Manin connection 
\ga{}{\nabla: H^2(\Omega^{\ge 1}_{\hat Y/\hat \Sigma}) \to
\Omega^1_{\hat \Sigma /W}  \otimes H^2_{DR}(\hat Y/\hat \Sigma) \notag}
is spanned by $c_1^{DR}(\sL)$ over $W$.

\end{theorem}
From Theorem~\ref{cor}  one easily deduces Theorem~\ref{thm:one}. 

\begin{remark} \label{rmk:follow}
Given $X$ as in Theorem~\ref{thm:one}, and $L\in \Pic(X)$ as in Proposition~\ref{existence}, then  $X_R\to \Spec R$  of  Theorem~\ref{thm:one} is constructed in such a way   that $L$ lifts to $X_R$. \end{remark}
\section{No lifting of automorphisms}\label{sec:Nolifting}

 In this section, applying Theorem ~\ref{thm:one}, we construct a projective model 
 $X_R\to \Spec R$ 
 of a K3 surface $X$, with $K={\rm Frac}(R)$  of characteristic zero, for   which almost all automorphisms of $X$ are not geometrically liftable. 
  
\begin{theorem} \label{cor:entropy2}
Let $X$ be a K3 surface defined over
 an algebraic closed field  $k$ of  characteristic 
$p >0$,  where $p>2$ if $X$ is  Artin-supersingular. 
\begin{itemize}
\item[1)]  Assume  that either the Picard number of $X$ is $\ge 2$ or that ${\rm Pic}\, (X) = \Z \cdot  H$ and $H^2\neq 2$.
Then there is a DVR $R$, finite over $W(k)$,  together with a projective model $X_R\to \Spec R$ of $X\to \Spec k$ such that no subgroup $G \subset {\rm Aut}\, (X)$, except for 
$G = \{ {\rm id}_X \}$,  is geometrically liftable to $X_R\to \Spec R$;
\item[2)]  Assume  that
${\rm Pic}\, (X) = \Z \cdot  H$ and  $(H^2) = 2$.   Then, for any projective   model $X_R\to \Spec R$ with $R$ finite over $W(k)$,   ${\rm Aut}^e(X_R/R)={\rm Aut}(X_R/R)$, 
the specialization homomorphism $\iota: {\rm Aut}(X_{\bar K}) \to {\rm Aut}(X)$ is an isomorphism, and ${\rm Aut}(X)=\Z/2$. 

\end{itemize}
\end{theorem}

\begin{proof}  First we prove 2).
 By replacing $H$ by $-H$ if necessary, we may assume that $H$ is ample. 
By Proposition ~\ref{existence} or, if $p \neq 2$, simply by $H^2(X, \Omega^2_{X/k}) \ni c_1^{{\rm Hodge}}(H)^{\cup 2}=$ residue class of $2$ in $k$, which is  thus non-zero, 
 $c_1^{\rm Hodge}(H) \not= 0$ and $H$ extends to a line bundle 
$H_R$ for any projective lift 
$X_R \to {\rm Spec}\, R$ to characteristic $0$. By \cite{SD74}, $h^0(X, H) = 3$, $h^i(X, H) = 0$ ($i \ge 1$), and $H$ is globally generated.  Strictly speaking,  $p \not= 2$ is assumed  in  \cite{SD74}. However, since $H$ is an ample generator, $h^0(H) \ge 3$ by the Riemann-Roch theorem. Moreover, any element in $\vert H \vert$ is irreducible and reduced. (Indeed, if  $C + D \in |H|$ for some non-zero effective divisors $C, D$, with possibly $C = D$, then $C \in |nH|$ and $D \in |mH|$ for some positive integers $n$, $m$ by ${\rm Pic}\, (X) = \Z\cdot  H$. 
However, then 
$2 \le n+m =1$, a contradiction.) So, we can apply \cite[Prop.~2.6, Thm.~3.1]{SD74}, which is characteristic free, to our $X$, to conclude that  $\vert H \vert$ defines a finite surjective morphism $\varphi : X \to \P^2_k$ of degree $2$. This is also separable even for $p=2$, as $X$ has no non-zero vector field by \cite{RS81}. We denote the covering involution by  $\iota\in {\rm Aut}\, (X)$. Since $h^i(X, H) = 0$ ($i \ge 1$), it follows that $H^0(X_R, H_R)$ is a rank $3$  free module over $R$, which satisfies base change. It thus defines a finite  surjective morphism $\varphi_R: X_R\to \P(H^0(X_R, H_R)^\vee)\cong \P^2_R$, of degree $2$, the specialization  of which over ${\rm Spec}\, k$ is $\varphi : X \to \P^2_k$.  We denote by   $\iota \in {\rm Aut}\, (X_R/R)$  the covering involution of $\varphi_R$, which exists as $\varphi_R$ is finite. Then  $\iota_R$ specializes to $\iota$.
 Hence, in the exceptional case ${\rm Pic}\, (X) = \Z \cdot  H$ with $(H^2) = 2$, the involution $\iota \in {\rm Aut}\, (X)$ lifts as an automorphism to any projective lift to characteristic $0$, in particular, $\iota$ is  geometrically liftable to any projective lift to characteristic zero.   More precisely, ${\rm Pic}\, (X_{\bar K}) = \Z \cdot  H_{\bar K}$, $(H_{\bar K}^2) = 2$ and ${\rm Aut}\, (X_{\bar K}) = \langle \iota_{\bar K} \rangle \simeq \Z/2$, where $\bar K$ is an algebraic closure of $K = {\rm Frac}\, (R)$ and $\iota_{\bar K}$ is the covering involution 
of the morphism $\varphi_{\bar K} : X_{\bar K} \to \P_{\bar K}^2$ given by $|H_{\bar K}|$.  This finishes the proof of 2).
\end{proof}
\begin{proof} We prove 1).
The following lemma  ought to  be well-known to the experts:
\begin{lemma}\label{trivial}
Let $Z$ be a complex projective K3 surface, with Picard group ${\rm Pic}\, (Z)$  generated by an ample class $H$. Assume that $(H^2) \not= 2$. Then ${\rm Aut}\, (Z) = \{\rm id_Z \}$. 
\end{lemma}

We do not know whether Lemma~\ref{trivial}  holds also for a K3 surface $X$ with $\rho(X) = 1$ in characteristic $p \ge 3$. However, if $X$ is a K3 surface defined over $\bar{\F}_p$, with $p\ge 3$, the positive solution to the Tate conjecture (\cite{MPe13}, also \cite{Ben14} with references therein) implies that  $\rho(X)$ is  even (\cite{Ar74}).

\begin{proof} We give a proof for the reader's  convenience. 

Since $H^0(Z, T_Z) = 0$ and ${\rm Aut}\, (Z)$ preserves the ample generator $H$, it follows that ${\rm Aut}\, (Z)$ is of dimension $0$ and also a closed algebraic subgroupscheme of the affine groupscheme ${\rm Aut}\, (\P (H^0(Z, mH)^\vee))$ for large $m >0$. It follows that ${\rm Aut}\, (Z)$ is a finite group. 

Let $T(Z)$ be the transcendental lattice of $Z$, that is, the orthogonal complement of ${\rm NS}\, (Z)$ in $H^2(Z, \Z(1))$. Thus the representation ${\rm Aut}\, (Z) \to {\rm O}\, ({\rm NS}\, (Z)) \times {\rm O}\,  (T(Z))$ has image in $\{1\}\times {\rm O}\,  (T(Z))$. On the other hand, since $Z$ is projective and ${\rm Aut}\, (Z)$ is a finite group,  by Nikulin \cite{Ni79}, the image of the natural map $G \to {\rm O}\, (T(Z))$ 
is a cyclic group of finite order $N$,  where  $N$ is exactly the order of  the image of the  representation ${\rm Aut}\, (Z) \to {\rm GL}\, (H^0(Z, \Omega_Z^2)) = \C \omega_Z,$ and  the Euler function $\varphi (N) := [\Q (e^{2\pi i/N}) : \Q]$ of $N$ 
is a divisor 
of the rank of the transcendental lattice $T(Z)$.  In our case, $T(Z)$  is of rank $21 = 22 -1$.  In particular, it is an odd number. Hence $N = 1$ or $2$. As $({\rm NS}(X)\oplus T(Z))\otimes_{\Z} \Q=H^2(Z, \Q(1))$ and ${\rm Aut}\, (Z)$ stabilizes  the lattice 
${\rm NS}( Z)\oplus T( Z)\subset H^2( Z, \Q(1))$, the image of the representation ${\rm Aut}\, (Z)\to GL(H^2(Z, \Q(1)))$ is cyclic of order $1$ or $2$, as well as the image of the representation $G\to GL(H^2( Z, \Z(1))$. 
 On the other hand, by the global Torelli theorem for complex projective K3 surfaces (\cite{PS71}), the action of ${\rm Aut}\, (Z)$ on $H^2(Z, \Z(1))$ is faithful. Thus ${\rm Aut}\, (Z)$ is either $\{{\rm  id}_Z \}$ or cyclic  of order $2$.

 So far, we did not use $(H^2) \not= 2$. Set $2d = (H^2) \ge 2$.  Then  by \cite[Cor.~1.6.2]{Ni79-2}, we have
$$\Z/2d \simeq {\rm NS}\, (Z)^*/{\rm NS}\, (Z) \simeq T(Z)^*/T(Z)\,\, ,$$
as $H^2(Z, \Z)$ is free and unimodular  and ${\rm NS}\, (Z)$ is primitive in $H^2(Z, \Z)$. Here $(-)^*$ means ${\rm Hom}_{\Z}((-), \Z)$  and $T(Z)=(\Z\cdot H)^\perp$.

As $\iota^*H$ is ample, and  $H$ is the unique ample generator of ${\rm Pic}\, (Z)$, one concludes that $\iota^*(H)=H$, thus 
$\iota^*={\rm id} $  on ${\rm NS}\, (Z)^*/{\rm NS}\, (Z)$. 
Hence $\iota^* = {\rm  id}$ on $T(Z)^*/T(Z)$ as well. On the other hand, the involution $\iota$ satisfies $\iota^* \omega_Z = -\omega_Z$ by Nikulin's result above. Thus 
$\iota^* = - {\rm id}$ on $T(Z)^*/T(Z)$. Thus ${\rm id} = -{\rm id}$ on $\Z/2d$. Hence $2d = 2$ as claimed.  \end{proof}

\begin{lemma}\label{trivial2}
Let $X$ be as in Theorem ~\ref{cor:entropy2}.  Assume that ${\rm Pic} (X) = {\rm NS}\, (X)$ 
is not isomorphic to $\Z\cdot  H$ with self-intersection number $(H^2) = 2$.  Then, there is an ample primitive line bundle $L$ such that $c_1^{\rm Hodge}(L) \not= 0$ and $(L^2) \not= 2$. 
\end{lemma}

\begin{proof} By Proposition ~\ref{existence}, there is an ample primitive line bundle $L_0$ such that $c_1^{\rm Hodge}(L_0) \not= 0$. If $(L_0^2) \not= 2$, then we may take  $ L=L_0$. In particular, if $\rho(X) = 1$, then we are done, as we exclude the case ${\rm Pic}\, (X) = \Z\cdot L$ with $(L^2) = 2$. 

So, we may assume without loss of generality that $\rho(X) \ge 2$ 
and $(L_0^2) = 2$. Since $L_0$ is primitive and $\rho(X) \ge 2$, we can choose a line bundle $M$ such that 
$\{L_0, M\}$  is   part of  a $\Z$-basis 
of ${\rm Pic}\, (X)$. Replacing $M$ by $M + nL_0$ with large integer $n$, we may further assume, without loss of generality, that $M$ is also ample. Note here that $\Z \langle L_0, M \rangle = \Z \langle L, M +nL_0 \rangle$ and $L_0$ and $M$ remain part of free $\Z$-basis under this replacement. Now consider $L= pM + L_0$. Then $L$ is ample, as $M$ and $L_0$ are ample. $L$ is also primitive, as $L_0$ and $M$ form part of free $\Z$-basis of $\Z$-module ${\rm Pic}\, (X)$. Moreover, 
$$(L^2) = p^2(M^2) + 2p(M.L_0) + (L_0^2) > 2\,\, ,$$
by $(M^2) > 0$, $(M.L_0) >0$ and $(L_0^2) >0$ by the ampleness. Thus $(L^2) \not= 2$.  Moreover,
$$c_1^{\rm Hodge}(L) = pc_1^{\rm Hodge}(M) + c_1^{\rm Hodge}(L_0) = 0 + c_1^{\rm Hodge}(L_0) = c_1^{\rm Hodge}(L_0) \not= 0\,\, ,$$
in the $k$-vector space $H^1(X, \Omega_X^1)$. So, 
$L = pM +L_0$ satisfies all the requirements.
\end{proof}

Now we are ready to finish the proof of Theorem~\ref{cor:entropy2} 1).
We take the model $X_R\to \Spec R$ of Theorem~\ref{thm:one}, Remark~\ref{rmk:follow}, applied to the ample primitive line bundle $L$ in Lemma ~\ref{trivial2}.  Assume that $G$ is geometrically liftable to $X_R\to \Spec R$. Then $G$  
 has to stabilize $\Pic(X_{\bar K})$, thus it fixes the  polarisation $\sL_{\bar K}$. Hence $G$ is a finite group. So we may  assume that there is an abstract field isomorphism $\bar K\rightarrow \C$ and $X_{\bar K}$ is a complex projective K3 surface, say $Z$, with Picard group ${\rm Pic}\, (Z)$  generated by an ample class $H_Z$ with $(H_Z^2) \not= 2$, and $G$ is now a group of automorphisms of $Z$. Then $G = \{{\rm id}_{X}\}$ 
by Lemma ~\ref{trivial}. 

\end{proof}

\section{Non-liftable automorphism of positive entropy}\label{sec:NonLiftable}

The aim of this section is to construct an example  of an automorphism of a supersingular  K3 surface over an algebraically closed  field 
$k$ of characteristic $p \ge 3$,
which is not geometrically liftable to any projective model.  As far as we are aware of, this is  the first such example. 
 Our construction is based on the work by Kondo-Shimada \cite{KS12}, 
and is (mildly) computer supported. The characteristic $p$ will be equal to $3$.

\medskip

Recall that ${\rm det}\, {\rm NS}\, (X) = -p^{2 \sigma_0}$ for a supersingular K3 surface defined over $k$. The value $\sigma_0$ is called the {\it Artin invariant} of $X$. Artin \cite{Ar74} proved that $1 \le \sigma_0 \le 10$ and Ogus \cite{Ogu79} proved the uniqueness of a supersingular K3 surface over $k$ with $\sigma_0 = 1$,  up to isomorphisms. Then  this K3 surface  is isomorphic to the Kummer K3 surface ${\rm Km}\, (E \times_k E)$ associated to the product abelian surface $E \times_k E$ of any supersingular elliptic curve $E/k$. Tate and Shioda (\cite{Sh75}) proved that the Fermat quartic K3 surface is supersingular if and ony if $p \equiv 3\ {\rm mod}   \ 4$. There are several other descriptions of supersingular K3 surfaces of Artin invariant $1$ (see e.g. \cite{Sh13}).

\medskip

{\it From now until the end of this section, $X$ is a supersingular K3 surface, defined over $k$ of characteristic $3$, with Artin invariant $1$.} As remarked above, $X$ is isomorphic to the Fermat quartic K3 surface.
 We denote  by $q : X \hookrightarrow \P^3$ the projective embedding and set $H=q^*\sO_{\P^3}(1)$.
 
 \medskip

In \cite{KS12},  Kondo and Shimada prove the following  statements, which are crucial for our construction: 
\begin{itemize}
\item[(i)]$X$ has two  globally generated line bundles $L_i$ ($i = 1$, $2$) of degree $2$ (${\sL_{m_i}}$ in their notation  \cite[p.~19]{KS12}).  They are not ample. The linear system associated to $H^0(X, L_i)$ induces a well defined morphism $\varphi_i: X\to \P^2$ which is generically finite $2:1$, but not finite.  The Galois involution of the function field extension extends as  an automorphism $\tau_i \in {\rm Aut}\, (X/\P^2)$. Indeed, it clearly extends as an automorphism in ${\rm Aut}(Y_i/\P^2)$ where $\varphi_i: X\to Y_i\to \P^2$ is the Stein factorization, and on the other hand, $X$ is the minimal desingularization of  the surface $Y_i$.

\item[(ii)] Let ${\rm Aut}\, (X, H)$ be the automorphism group of $X$ induced by the projective linear automorphisms of $\P^3$ under $q$. It is known that ${\rm Aut}\, (X, H)$ is a finite group but  of huge order (see. e.g. \cite{DK09}, \cite{Mu88}). In ${\rm Aut}\, (X, H)$, there is a special element $\tau \in {\rm Aut}\, (X, H)$  of order $28$ (\cite[Ex.~3.4]{KS12}). 
\end{itemize}

They prove the following beautiful description of the automorphism group of $X$:

\begin{theorem} \label{KondoShimada}
${\rm Aut}\, (X) = \langle \tau_1, \tau_2, {\rm Aut}\, (X, H) \rangle$.
\end{theorem}

In the course of the proof, they  work with an explicit $\Z$-basis $\sB$ of ${\rm NS}\, (X)$, consisting of $22$ lines
among the 112 lines on $X$  (\cite[Lem. 6.3]{SSL10}) and compute the (right, hence covariant) representation of $(\tau_i)_* \vert_{ {\rm NS}\,(X)}$, $\tau_* \vert_{ {\rm NS}\,(X)}$ on ${\rm NS}\, (X)$. They actually write explicitely  the matrices  of $(\tau_i)_* \vert_{ {\rm NS}\,(X)}$ and $\tau_* \vert_{ {\rm NS}\,(X)}$  in the basis $\sB$.  We denote them by $A_1$, $A_2$, $T$  respectively, of which explicit forms are in  Tables 5.4, 5.5, 3.3 in \cite{KS12}. These forms are important  for the proof of Theorem \ref{KondoShimada}.

Suppose one has a model $X_R\to {\rm Spec} \ R$ on which $L_i$ lift to $L_{i, R}$. Then, as $H^0(X_R, L_{i, R})\otimes_R k=  H^0(X, L_i)$, one has base change $Y_{i, R}\otimes_R k=Y_i$ for the Stein factorization 
$\varphi_{i,R}: X_R\to Y_{i, R}\to \P^2_R$ 
of the well defined morphism $\varphi_{i, R}: X_R\to \P^2_R$ associated to $H^0(X_R, L_{i, R})$. 
The Galois involution of the function field extension induced by $\varphi_{i, R}$ extends as an automorphism 
$\tau_{Y_{i, R}} \in {\rm Aut}(Y_{i,R}/\P^2_R)$, which induces a birational automorphism  $\tau_{i,R}\in 
{\rm Bir}(X_R/\P^2_R)$. One says that $\tau_i$ {\it lifts to } $X_R$ if   $\tau_{i,R}\in {\rm Aut}(X_R/\P^2_R)\subset 
{\rm Bir}(X_R/\P^2_R)$.

\begin{theorem} \label{NoLiftTwoDegreeTwo}
There is a projective model $X_R\to \Spec \ R$  of $X = X(3)$ with $R$ of characteristic $0$ such that ${\rm Aut}^e(X_{\bar K}) \not= {\rm Aut}\, (X_{\bar K})$.  Here $K = {\rm Frac}(R)$ and $\bar K$ is an algebraic closure 
of $K$. 
\end{theorem}

\begin{proof}  Let $H$ be an ample line bundle on $X$. By \cite[App.~A]{LieOls11}, there is a model $X_R\to \Spec \ R$ such that $L_1$, $L_2$ and $H$ lift.  This property is compatible with further base change $R\subset R_L$ where $L\supset K$ is a field extension with $L\subset \bar K$.
This model is projective. Let us denote by $f_i$ the restriction of $\tau_{i,R}$ to $X_{\bar K}$. Then $f_i \in {\rm Aut}(X_{\bar K})$ as $X_{\bar K}$ is a minimal smooth projective surface. It suffices to show that one of $f_i$ ($i =1$, $2$) is not in ${\rm Aut}^e(X_{\bar K})$. Assume to the contrary that both $f_i$ ($i =1$, $2$) are in ${\rm Aut}^e(X_{\bar K})$.
 Then   $ f_i^*\omega = -\omega,$ where $\omega$ is a non-zero  global $2$-form of $X_{\bar K}$. Then  $(f_1 \circ f_2)^*\omega = \omega$.  Thus $(f_1\circ f_2)^*$ has one  eigenvalue equal to one on de Rham cohomology $H^2_{DR}(X_{\bar K}/{\bar K})$, thus, by the comparison theorem, on $\ell$-adic cohomology $H_{{\rm \acute{e}t}}^2(X_{\bar K}, \Q_\ell(1))$ as well, thus by  (\cite[V,~Thm.~3.1]{SGA4.5}), on $H_{{\rm \acute{e}t}}^2(X, \Q_\ell(1))$ as well. 

On the other hand, using the explicit forms of 
$A_1$ and $A_2$  in Tables 5.4, 5.5 in \cite{KS12}, and {\tt Mathematica  (all we need here are Dot command, CharacteristicPolynomial command, Factor command}), we find that the characteristic polynomial of $(\tau_2 \circ \tau_1)^* \vert_{ {\rm NS}\,(X)}  = \tau_{1*} \circ \tau_{2*}|_{NS(X)}$, i.e., of $A_1A_2$, is
$$(1 + x + x^2) (1 - 11 x + 10 x^2 - 9 x^3 + 9 x^4 - 10 x^5 + 15 x^6 - 
   23 x^7 + 19 x^8 - 14 x^9 + 14 x^{10} - 14 x^{11} + 19 x^{12}$$
$$ - 23 x^{13} + 
   15 x^{14} - 10 x^{15} + 9 x^{16} - 9 x^{17} + 10 x^{18} - 11 x^{19} 
+ x^{20})\,\,,$$
of which  $1$ is not a zero,  a contradiction. 
\end{proof}

\begin{remarks} \label{NoLiftAlg} 
 On the model of Theorem~\ref{NoLiftTwoDegreeTwo}, one of the $\tau_i$ does not lift, which means that a small modification occurs on the special fiber.  In fact one can say the following. 
Given a  projective model $X_R\to {\rm Spec}  \ R$, with ample line bundle $L_R$, then $\tau_i$  always lifts to an automorphism $\tau_i^0$  of $X_R^i:=X_R\setminus \Sigma_i$, where $\Sigma_i$  is  the exceptional  locus of $X\to Y_i$.   
So one can always define the line bundle $L_{i, R}= ( \tau_i^0)^*L_R|_{X_R^i} \in {\rm Pic} (X_R^i)={\rm Pic}(X_R)$. Then $\tau_i$ lifts to $X_R$ if and only if 
$L_{i,R}$ is ample, which is equivalent to
$L_{i, R}\otimes_R k= \tau_i^* (L_R\otimes_R K)$.

\end{remarks}
The next theorem gives the (first)  explicit example of a positive entropy  automorphism of  a K3 surface which 
is not geometrically liftable  to characteristic zero. 
\begin{theorem} \label{NoLiftPosEntropy} The automorphism
$f := \tau_1 \circ \tau \circ \tau_2 \circ \tau \in {\rm Aut}\, (X)$
\begin{itemize}
 \item[(1)] is not geometrically liftable to any projective model $X_R\to \Spec R$ of characteristic zero  as well as its power $f^n$ 
($n \in \Z \setminus \{0\}$); 
 \item[(2)] has positive entropy $h(f)$ equal to the logarithm of a Salem number  $a$ of degree $22$;
 \item[(3)] $h(f)$ is not the entropy of any automorphism on any  projective K3 surface  in characteristic $0$;
 \item[(4)] numerically
$$a =  26.9943 \ldots\,\, ,\,\, h(f) = \log\, 26.9943 \ldots\,\, .$$
\end{itemize}
\end{theorem}

\begin{proof} Using the explicit forms  of $A_1$, $A_2$, $T$ in Tables 5.4, 5.5, 3.3 in \cite{KS12}, and {\tt Mathematica (all we need here are Dot command, CharacteristicPolynomial command, Factor command, and NSolve command)}, we find that the characteristic polynomial  $P$ of 
$f_* \vert_{ {\rm NS}\,(X))}$, 
i.e., of $A_1TA_2T$, is
$$1 - 27 x + 4 x^3 + 3 x^4 + 24 x^5 + 15 x^6 - 7 x^7 + x^8 - 14 x^9 - 
 2 x^{10} - 5 x^{11} - 2 x^{12}$$
$$ - 14 x^{13} + x^{14} - 7 x^{15} + 15 x^{16} + 
 24 x^{17} + 3 x^{18} + 4 x^{19} - 27 x^{21} + x^{22}\,\, .$$
This is an irreducible Salem polynomial of degree $22$. 
The fact that this is irreducible is checked by {\tt Factor command}. Then this is a Salem polynomial by Proposition~\ref{salem}. Indeed, it is either a Salem polynomial or a cyclotomic polynomial of degree $22$. But {\tt NSolve command} shows that one of  the zeroes of the above polynomial is approximately $26.9943$. Hence it is a Salem polynomial of degree $22$ with Salem number approximately $26.9943$. This shows (2). 

By \cite[Lem.~2]{Smy14}, if  $\lambda$ is a Salem number of degree $d$, then $\lambda^n$ is a Salem number of the same degree $d$ for all $n \in \N \setminus \{0\}$.  Applying this to the eigenvalue $\lambda$  of $f$ or  of $f^{-1}$ which is a Salem number, one concludes that $f^n$,  for all $n\in \Z\setminus \{0\}$,  has entropy the logarithm of a Salem number of degree $22$.

As in characteristic $0$, an automorphism  $g$ always stabilizes $NS(X)\subset H^2(X_{\bar K}, \Q_\ell(1))$ and the Picard rank is at most $20$, the logarithm of the absolute value of a root of $P$  can not be
the entropy of $g$.  This shows (3).

If $f^n, \ n\in \Z$ was lifting to an automorphism $g$ on  $X_R\to \Spec  R$, then the specialization $\iota: NS(X_{\bar K})\hookrightarrow NS(X)$  (see Section~\ref{ss:K3}) would be $g$ equivariant, thus, as the Picard rank of $X_{\bar K}$ is at most $20$, the minimal polynomial of $g$ could not have degree $22$. This shows (1) and finishes the proof. 
\end{proof}
\begin{remarks} \label{rmk:hodge}
1) As discussed in \cite[Conj.~1.2]{BlEsKe14}, one expects that the rational crystalline cycle class of an algebraic cycle, expressed as a de Rham class on a model in characteristic $0$, is the cycle class of an algebraic cycle on the model, if and only if it is in the right level of the Hodge filtration.  For the cycle class  $c$ of the graph of an automorphism $f$ on a K3 surface, the conjecture is verified (as written in  \cite[Cor.~2.5]{Ogu79}). Indeed,  
$$c\in F^2H^4_{DR}(X_R\times_R X_R/R)=F^2H^4(\hat X_R\times_R \hat X_R/R)$$  
if and only if $f^*$ acting on $H^2_{DR}(\hat X_R/R)=\varprojlim_n H^2_{DR}(X_n/(R/ \langle \pi^n \rangle))$ respects the Hodge filtration.  Here $\pi$ is the uniformizer of $R$ and $X_n=X\otimes_R R_n, \ R_n= R/\langle \pi^n\rangle$. 
Clearly, if $f$ lifts, then $c$ is the cycle class of the graph and lies in $F^2H^4(\hat X_R\times_R \hat X_R/R)$.
Let us now assume that $c\in F^2H^4(\hat X_R\times_R \hat X_R/R)$. 
 The obstruction to lifting $f_n$ on $X_n$ to $f_{n+1}$ on $X_{n+1}$  lies in $H^1(X_n, f_n^*T_{X_n/ R_n} \otimes \pi^n|_{X_n})$ and is identified with the action of $f_n^*$ in  ${\rm Hom} (H^1(f^*(\Omega^1_{X_n/R_n} )  ), \pi^n \otimes H^2(\sO_{X_n}))$, thus dies. One constructs  in  this way a prosystem of lifts $\varprojlim_n f_n$, thus, a formal scheme $\varprojlim_n \Gamma_n$, where $\Gamma_n\subset X_n\times_{R_n} X_n$ is the graph of $f_n$, thus, by  \cite[Chap.III,~Thm.~5.4.5]{EGA3},  a projective scheme $\Gamma_R\subset X_R\times_R X_R$ which lifts the graph of $f$ and thus defines the lift. 

So the test whether or not  an automorphism lifts to characteristic $0$ is of $p$-adic nature.
On the other hand,  the test we develop  in  Theorem~ \ref{NoLiftPosEntropy} relies on the degree of an  algebraic integer.   It is of course very specific to our situation, nonetheless it is intriguing.
 
 \medskip
 
 2) The Salem number we define in Theorem~ \ref{NoLiftPosEntropy} does not come from  the entropy  of an automorphism on a  projective K3 surface in characteristic $0$. One could perhaps speculate that there is a projective model $V_R\to \Spec R$ of a higher dimensional smooth projective variety $V_K$ in characteristic $0$, with an automorphism $f_R$ of $V_R/R$, such that its entropy is reached on the class of a $1$-cycle, the support of which,  by specialization,  lies on the K3 surface considered in Theorem~ \ref{NoLiftPosEntropy}, as a higher codimensional cycle on $V_R\otimes_R k$.  Though we do not have any computation going in this direction, this would just be nice. 

\end{remarks}

\section{Lifting of automorphisms of supersingular  K3 surfaces of Artin invariant $1$ in large characteristic}
 \label{sec:EllipticK3}
Throughout this section, $k$ is
 an algebraically closed  field of  characteristic 
$p  \ge 3$
and $X=X(p)$
as in Section~\ref{sec:NonLiftable}. So $X(p)$ is a supersingular K3 surface of Artin invariant $1$, and is uniquely defined up to isomorphism with this property. 

Using Ogus' crystalline Torelli theorem \cite{Ogu83}, J. Jang \cite[Thm.3.3]{Jan14} proved the following theorem:
\begin{thm} \label{thm:jang}
The image of the representation of ${\rm Aut}(X)$ in the linear automorphism of the one dimensional vector space $k\cdot \omega= H^0(X, \Omega^2_{X/k})$ is a cyclic group of cardinality $p+1$. 

\end{thm}
We denote by  $h$ an element in ${\rm Aut}(X)$ such that $h^*\omega=\xi_{p+1} \omega$ where $\xi_{p+1}$ is a primitive $(p+1)$-th root of unity in $k$. 

\begin{remark} \label{rmk:jang}
Let $M$ be a projective K3 surface over a characteristic $0$ field $K$. Then the image of ${\rm Aut}(M)$ 
 in the linear automorphism of the one dimensional vector space $K\cdot \omega= H^0(M, \Omega^2_{M/K})$ 
is a cyclic group of order $\le 66$ (\cite{Ni79}). In fact Nikulin considered the image of a finite subgroup, but, by the finiteness of the pluri-canonical representation in characteristic $0$ (Ueno-Deligne,  \cite[Thm.~14.10]{Ue75}),  the proof extends to the whole automorphism group.
\end{remark}

Jang deduces from Theorem~\ref{thm:jang}  and Remark~\ref{rmk:jang}  the following:

\begin{corollary} \label{cor:jang}
If $p\ge 67$, then $h$ is not geometrically liftable to characteristic $0$.

\end{corollary}

The aim of this section is to construct an element $\tau \in {\rm Aut}(X)$ {\it of positive entropy} which is not geometrically liftable to characteristic $0$.  (Recall $X=X(p)$).

\begin{definition} \label{def:number}
We define $\beta$ to be the least common multiple of the natural numbers $n$ such that the value of the Euler function $\varphi(n)$ is smaller or equal to  $22$. 

\end{definition}
\begin{thm} \label{thm:beta}
If $p+1\ge  67 \beta$, then there is an automorphism  $\tau $ in  $ {\rm Aut}(X(p))$, of positive entropy, which is not geometrically  liftable to characteristic $0$.

\end{thm}
In order to prove the theorem, we first state the following lemmata.

For an endomorphism $\theta$ of  a free  $\Z$-module  of finite type, we denote by ${\rm tr} (\theta)$ its trace with values in $\Z$.

\begin{lemma}\label{trace} Let $X$ be a K3 surface over any field.
Let $f \in {\rm Aut}\, (X)$ such that  $$\vert {\rm tr}\, (f^*  |_{{\rm NS}\, (X)}) \vert \ge 23.$$ Then $f$ is of positive entropy. Conversely, if $f$ is of positive entropy, then there is a positive integer $N$ such that 
$\vert {\rm tr}\, ((f^{n})^{*} \vert _{{\rm NS}\, (X)}) \vert \ge 23$ for all 
integers $n$ such that $n \ge N$.
\end{lemma}
\begin{proof} Recall that ${\rm NS}\, (X)$ is of rank $\rho \le 22$. Let $\alpha_i$ ($1 \le i \le \rho$) be the eigenvalues of $f^* \vert_{ {\rm NS}\, (X)}$. If $f$ is not of positive entropy, then $\alpha_i$ are cyclotomic integers by 
Remark~\ref{entropy}. Hence
$$\vert {\rm tr}\, (f^* \vert _{{\rm NS}\, (X)}) \vert = \vert \sum_{i=1}^{\rho} \alpha_i \vert \le 
\sum_{i=1}^{\rho} \vert \alpha_i \vert = \rho < 23\,\, .$$
Hence $f$ is of positive entropy if $\vert {\rm tr}\, (f^* \vert_{{\rm NS}\, (X)}) \vert \ge 23$. Assume that $f$ is of positive entropy. Then, after renumbering,  $\alpha_1$ is a Salem number $a >1$, $\alpha_2$ is $1/a$ and all other $\alpha_k$ are of absolute value $1$ by Remark~\ref{entropy}. Then 
$${\rm tr}\, ((f^{N})^{*} \vert _{{\rm NS}\, (X)}) = a^{N} + \frac{1}{a^N} + \sum_{k=3}^{\rho} \alpha_k^{N}\,\, ,$$
which is an integer, in particular, real. Hence
$${\rm tr}\, ((f^{N})^{*} \vert _{{\rm NS}\, (X)}) = a^{N} + \frac{1}{a^N} + \sum_{k=3}^{\rho} {\rm Re}\,(\alpha_k^{N})\,\, .$$
Since $\vert \alpha_k^N \vert = \vert \alpha_k \vert^N = 1$, it follows that 
$$-18 = \sum_{k=3}^{\rho} -1 \le \sum_{k=3}^{\rho} {\rm Re}\,(\alpha_k^{N}) \le \sum_{k=3}^{\rho} 1 = 18\,\, .$$
On the other hand, since $a > 1$, it follows that 
$$\lim_{N \to \infty} a^{N} + \frac{1}{a^N} = +\infty\,\, .$$
Hence, there is $N$ such that ${\rm tr}\, ((f^{n})^{*} \vert_{ {\rm NS}\, (X)}) \ge 23$, hence $\vert {\rm tr}\, ((f^{n})^{*} \vert_{ {\rm NS}\, (X)}) \vert \ge 23$, for all integers $n$ such that $n \ge N$. 
\end{proof}

\begin{lemma}\label{positiveentropy} 
Let $X=X(p)$. Then
there is $f \in {\rm Aut}\, (X)$ such that $\vert {\rm tr}\, (f^* \vert_{ {\rm NS}\,(X)}) \vert > 23$,  and $f^*\omega=\omega$. In particular, $f$ is of positive entropy.
\end{lemma}

\begin{proof}  
Recall that $X(p)={\rm Km}(E\times_k E)$ where $E$ is a supersingular elliptic curve. 
The projections $pr_i: E\times_k E\to E, \ i=1,2$  descend to elliptic fibrations $\varphi_i: X\to \P^1\cong E/\langle \pm 1 \rangle$ (thus with section) and with exactly $4$ singular fibers of type $I_0^*$. 
We choose  a zero-section of $\varphi$ and denote it by $O$. We identify it with its image $O \subset X$. Let ${\rm MW}\, (\varphi_i)$ be the Mordell-Weil group of $\varphi_i$. It acts  on the generic fiber $X_\eta$ of $\varphi$ by translation, thus is an abelian  subgroup of ${\rm Aut}(X)$. 

One has (\cite{Sh90})

\ga{}{{\rm rank} \ {\rm MW}(\varphi_i)={\rm rank} \ {\rm NS}(X) -2 -4\times 4= 4. \notag
}

For each $i = 1$, $2$, choose $f_i \in {\rm MW}\, (\varphi_i)$ such that $f_i$ is of infinite order. The N\'eron-Severi class $e_i$ of a closed fiber of $\varphi_i$ is stable under $f_i$, i.e.,  $f_i^*(e_i) = e_i$  for $i =1$, $2$, $e_1$,  $e_2$ are of self-intersection $0$,  and they are linearly independent in the hyperbolic lattice ${\rm NS}\, (X)$. Then, by \cite[Thm.~ 3.1]{Og09}, $f_3 := f_1^{n_1} \circ f_2^{n_2} \in 
{\rm Aut}\, (X)$ is of positive entropy for large positive integers $n_1$, $n_2$. 
Note that $f_i^*\omega=\omega$ as $f_i$ is a translation automorphism of an elliptic fibration. Thus $f_3^*\omega=\omega$ as well. 
Hence $f = f_3^N$ for large $N$ 
is a required solution. 
\end{proof}

\begin{proof}[Proof of Theorem~\ref{thm:beta}]
If $h$ is of positive entropy, we are done. So we assume $h$ has null-entropy. This means, the characteristic polynomial of $h^*|_{{\rm NS}(X)}$ is the product of cyclotomic polynomials of degree $\le 22$. Set $g:= h^\beta$. As ${\rm rank} \ {\rm NS}(X)= 22$,  $g^*|_{{\rm NS}(X)}$ is unipotent, so  there is a basis of ${\rm NS}\, (X) \otimes \Q$ in which $g^* |_{{\rm NS}(X)\otimes \Q}  $ is represented by the Jordan canonical form:
$$g^* |_{{\rm NS}(X)\otimes \Q}  = J := J(r_1, 1) \oplus \cdots \oplus J(r_s, 1)\,\, .$$
Set 
$s = {\rm max}\, \{r_i\}_{i=1}^{s} -1$ and choose $f$ as in Lemma~\ref{positiveentropy}.

 If $s=0$, then $g^*|_{{\rm NS}(X)}={\rm Id}$. Hence $(f\circ g)^*|_{{\rm NS}(X)} =f^*|_{{\rm NS}(X)}$ and $(f\circ g)^*\omega=\xi_m \omega$ with some $m\ge 67$. Thus $f\circ g \in {\rm Aut}(X)$ is of positive entropy and not geometrically liftable to characteristic $0$.

Assume $s>0$. 

In the same basis, $f^*|_{{\rm NS}(X)\otimes \Q}$ is represented by a matrix $A = (A_{ij})$ where 
$A_{ij}$ is $r_i \times r_j$ matrix located at the $(i, j)$-block. Then 
$${\rm tr}\, ((f \circ g^N)^* \vert _{{\rm NS}\,(X)}) = {\rm tr}\, (J^NA) = \sum_{i=1}^{s}{\rm tr}\, (J(r_i, 1)^NA_{ii}).$$ 
We want to estimate each summand of the formula above. 
For instance, if $r_1 = 4$, then 
$$J_4 := J(4, 1) = I_4 + R_4\,\, ,\,\, R_4 := \left(\begin{array}{rrrr}
0 & 1 & 0 & 0\\
0 & 0  & 1 & 0\\
0 & 0 & 0 & 1\\
0 & 0 & 0 & 0
\end{array} \right)\,\, ,$$
and $R^4 = 0$. Thus, by the binomial expansion
$$J_4^N = I_4 + N 
\left(\begin{array}{rrrr}0 & 1 & 0 & 0\\
0 & 0  & 1 & 0\\
0 & 0 & 0 & 1\\
0 & 0 & 0 & 0\end{array} \right) + \frac{N(N-1)}{2}
\left(\begin{array}{rrrr}
0 & 0 & 1 & 0\\
0 & 0  & 0 & 1\\
0 & 0 & 0 & 0\\
0 & 0 & 0 & 0\end{array} \right) + \frac{N(N-1)(N-2)}{6}
\left(\begin{array}{rrrr}
0 & 0 & 0 & 1\\
0 & 0  & 0 & 0\\
0 & 0 & 0 & 0\\
0 & 0 & 0 & 0\end{array} \right)\,\, .$$
Hence, for a $4 \times 4$-matrix $A_4 := (a_{ij})$ of rational entries, 
we have
$${\rm tr}\, (J_4^NA_4) = {\rm tr}\, (A_4) + N(a_{21} + a_{32} + a_{43}) + \frac{N(N-1)}{2}(a_{31} + a_{42}) + \frac{N(N-1)(N-2)}{6}a_{41}\,\, ,$$ 
which is a polynomial of $N$ of degree $\le 3 = 4-1$, with 
rational coefficients depending on $A_4$ (and independent of $N$).  
For exactly the same reason,  from the expansion of $J(r, 1)^N$ as above, 
we find that 
 ${\rm tr}\, ((f \circ g^N)^* \vert_{ {\rm NS}\,(X)})$ 
is of the form
$${\rm tr}\, ((f \circ g^N)^* \vert_{ {\rm NS}\,(X)}) = a_sN^s + a_{s-1}N^{s-1} + \cdots 
+ a_1 N + {\rm tr}\, A\,\, .$$
Here  $a_k$ are rational numbers depending only on $A$ (and independent of $N$). So, if it  happened that $\vert {\rm tr}\, ((f \circ g^{k_n})^* \vert_{ {\rm NS}\,(X)}) \vert \le 22$ for some sequence of positive integers 
$$k_1 < k_2 < k_3 < \cdots < k_n < \cdots\,\, \to +\infty\,\, ,$$
 one would have
$$a_s = a_{s-1} = \cdots = a_{1} = 0\,\, ,$$
and hence
$$\vert {\rm tr}\, ((f \circ g^N)^* \vert _{{\rm NS}\,(X)}) \vert = \vert {\rm tr}\, (A) \vert$$
for all positive integer $N$.  But
$$\vert {\rm tr}\, (A) \vert = \vert {\rm tr}\, (f^* \vert_{ {\rm NS}\,(X)})) \vert \ge 23\,\, ,$$
by our choice of $f$, a contradiction to $\vert {\rm tr}\, ((f \circ g^{k_n})^* \vert_{ {\rm NS}\,(X)})\vert \le 22$. 
Hence, there are only finitely many positive integers $\ell$ such that $\vert {\rm tr}\, ((f \circ g^{\ell})^* \vert_{ {\rm NS}\,(X))}) \vert \le 22$. Thus, there is a positive integer $M$ such that $\vert {\rm tr}\, ((f \circ g^N)^* \vert _{{\rm NS}\,(X)})) \vert \ge 23$ for all integers $N$ such that $N \ge M$.

Let us choose $N$ such that $N\ge M$ and $( N, m )=1$. Define $\tau:= f\circ g^N$. Then $|{\rm tr} (\tau^*|_{{\rm NS}(X)})| \ge 23$, and $\tau^*\omega=\xi_m^N \omega$.  Note $\xi_m^N$ is a primitive $m$-th root of unity with $m\ge 67$.  Thus $\tau\in {\rm Aut}(X)$ is of positive entropy and not geometrically liftable to characteristic $0$.

\end{proof}

\begin{remark}\label{independentoflift}
In \cite{EOY14} we show that various elliptic structures on $X(p)$, for 
 $p \ge 11$ and $p \not= 13$, with a rich Mordell-Weil group force the existence of autmorphisms with entropy the logarithm of a degree $22$ Salem number. 
\end{remark}

\end{document}